\documentclass[12pt]{amsart}

\usepackage{amsrefs}
\usepackage{amssymb}
\usepackage{hyperref}
\usepackage{fancyhdr}
\usepackage{xcolor}
\usepackage{bbm}

\newtheorem{theorem}{Theorem}
\newtheorem*{theorem*}{Theorem}
\newtheorem{lemma}[theorem]{Lemma}
\newtheorem{proposition}[theorem]{Proposition}
\newtheorem{claim}[theorem]{Claim}
\newtheorem{corollary}[theorem]{Corollary}

\theoremstyle{definition}
\newtheorem{definition}[theorem]{Definition}
\newtheorem{example}[theorem]{Example}
\newtheorem*{definition*}{Definition}
\newtheorem*{lemma*}{Lemma}
\usepackage{graphicx}
\usepackage{pstricks, enumerate, pst-node, pst-text, pst-plot}

\numberwithin{equation}{section}
\numberwithin{theorem}{section}

\newcommand{\R}{\mathbb{R}}

\renewcommand{\P}[1]{{\mathbb{P}}\left[{#1}\right]}

\newcommand{\CondP}[2]{{\mathbb{P}}\left[{#1}\middle\vert{#2}\right]}

\newcommand{\cS}{\mathcal{S}}
\newcommand{\cG}{\mathcal{G}}
\newcommand{\cC}{\mathcal{C}}

\newcommand{\cX}{\mathcal{X}}

\newcommand{\strats}{\mathcal{S}}

\newcommand{\mrf}{\mbox{MRF}(G,\strats)}
\newcommand{\pgg}{\mbox{PGG}(G,\strats)}

\newcommand{\dist}{\Delta}

\begin{document}

\title[]{Graphical potential games}

\author[]{Yakov Babichenko}
\address{California Institute of Technology}
\author[]{Omer Tamuz}


\date{\today}

\begin{abstract}
  We study the class of potential games that are also graphical games
  with respect to a given graph $G$ of connections between the
  players. We show that, up to strategic equivalence, this class of
  games can be identified with the set of Markov random fields on
  $G$.

  From this characterization, and from the Hammersley-Clifford
  theorem, it follows that the potentials of such games can be
  decomposed to local potentials. We use this decomposition to
  strongly bound the number of strategy changes of a single player
  along a better response path. This result extends to generalized
  graphical potential games, which are played on infinite graphs.
\end{abstract}

\maketitle
\tableofcontents

\section{Introduction.}

Potential games form an important class of strategic
interactions. They includes fundamental interactions such as Cournot
oligopolies (see, e.g.,~\cite{MS}), congestion games (see, e.g.,
Monderer and Shapley~\cite{MS} or
Rosenthal~\cite{rosenthal1973class}), routing games (see, e.g.,
Rosenthal~\cite{rosenthal1973class}) and many others. The above
mentioned interactions are frequently \emph{local} in nature. Namely,
there exists an underlying graph such that the payoff of a player
depends on her own strategy and on the strategies of her neighbors, but
does not depend on the strategies of the opponents who are not
neighbors. For instance, the locality of the interaction could be
geographical: In routing games, the outcome of a driver depends only
on her own route and the routes that were chosen by drivers who are
geographically close to her. In a Cournot oligopoly where
transportation costs are high, a firm is competing only with firms
which are geographically close (for instance, this is the case with
natural gas market; see Victor, Jaffe and
Hayes~\cite{victor2006natural}). The idea of the locality of an
interaction is captured by the notion of a \emph{graphical game},
introduced by Kearns et al.~\cite{kearns2001graphical}. These games
and similar ones are also sometimes called {\em network games}; see
Jackson and Zenou~\cite{jackson2012games} for an extensive survey.

The goal of this paper is to understand the class of graphical
potential games. First, we address the following questions: What
characterizes the potential function of a graphical potential game?
What characterizes the payoffs of the players in a potential graphical
game? In Theorems~\ref{claim:graphical-potential},~\ref{pro:graphical-equivalence} and~\ref{pro:payoffs} we provide a complete answer to these questions. We show that
\begin{enumerate}
\item The potential function of a graphical potential game can be
  expressed as an additive function of \emph{local potentials}, where
  each local potential corresponds to a maximal clique in the
  underlying graph, and the value of the local potential is determined
  by the strategies of the players in the maximal clique only. This
  condition is necessary for a potential game to be graphical, and
  every such potential is the potential of some graphical game.

\item Up to strategically equivalent transformations (see
  definition~\ref{def:strat-equ}), the payoff of a player is the sum
  of the local potentials of the cliques to which she belongs.
\end{enumerate}  

The proof of these results is achieved by showing that, for a fixed
graph $G$, the set of potentials of graphical games on $G$ can be
identified with the set of \emph{Markov random fields} on $G$ (see
Section~\ref{sec:mrfs}). The latter are a well studied class of
probability distributions with certain graphical Markov
properties. Having established this correspondence, the
Hammersley-Clifford theorem~\cite{hammersley1968markov}, a classical
result on Markov random fields, yields the above characterization of
graphical potential games.

Next, we use this characterization to study dynamics. A central class
of dynamics that has been studied in the context of potential games is
(strict) better-response dynamics, where at each period of time a
single player updates her strategy to a (strictly) better one, with
respect to the current strategies of the opponents; the basic
observation regarding all potential games is that better-response
dynamics always converge.

Such a sequence of unilateral improvements is called \emph{a better
  response path}. We address the question: Are there properties of
better response paths that are unique for \emph{graphical} potential
games? We focus on the \emph{number of updates of a single player}
along a better response path, and we prove in
Theorem~\ref{thm:num-changes} and Corollary~\ref{cor:constant} that
under mild uniformity assumptions on the game, and for graphs with
slow enough growth (see Definition~\ref{def:growth}), the number of
updates of each player is bounded by a constant, which in particular
is independent of the number of players. This result is general and
holds for all the above mentioned
examples. Example~\ref{ex:non-constant-changes} demonstrates that the
slow-enough-growth condition is tight, in some sense.

Finally, in Section~\ref{sec:generalizations}, we introduce {\em
  generalized graphical potential games} on infinite graphs. These are
graphical games which are not necessarily potential games, but still
have a local potential structure. We prove here the same bounds on
strict better-response paths. We note that these
results, as well as the results on finite graphs, translate to novel
results on Markov random fields, which may be of interest outside of
game theory.

\subsection{Related Literature.}

Potential games and graphical games are both fundamental classes of
games (see, e.g., Nisan~\cite{nisan2007algorithmic}). There has
recently been a growing interest in the intersection of these two
classes, because many interesting potential games have a graphical
structure (see e.g., Bil{\`o} et al.~\cite{bilo2011graphical} or
Bimpikis, Ilkilic and Shayan~\cite{BIS2014}), and many interesting
graphical games have a potential function (see e.g., Auletta et
al.~\cite{auletta2011convergence}, Bramoull{\'e}, Kranton and
D'amours~\cite{bramoulle2014strategic}). However, to the best of our
knowledge, this paper is the first to fully characterize the
intersection of these two classes\footnote{In an independent, later
  work Ortiz~\cite{ortiz2015graphical} addresses a similar
  characterization problem.}.

Best-response and better-response paths in potential games were
studied in Fabrikant, Papadimitriou and
Talwar~\cite{fabrikant2004complexity}, Skopalik and
V{\"o}cking~\cite{skopalik2008inapproximability} and Awerbuch et
al.~\cite{awerbuch2008fast}, where the focus is on the \emph{length}
of the paths. Our result (Theorem~\ref{thm:num-changes} and
Corollary~\ref{cor:constant}) focuses on the \emph{number} of changes
of a single player. This aspect of better- and best-response dynamics
also plays an interesting role in graphical games of strategic
complements; see Jackson and Zenou~\cite{jackson2012games}.

The connection between graphical potential games and Markov random
fields was previously observed for specific cases; see Auletta et
al.\cite{auletta2011convergence}, who establish a connection between
the logit dynamic in a particular class of graphical potential
coordination games and Glauber dynamics in the Markov random field of
the Ising model. We show that this connection is much more general and
extends to all graphical potential games.

Another intriguing connection between graphical games and Markov
random fields was established by Kakade et
al.~\cite{kakade2003correlated}, who show that every correlated
equilibrium of every graphical game is a Markov random
field. Daskalakis and Papadimitriou~\cite{daskalakis2006computing} use
Markov random fields to compute pure Nash equilibria in graphical
games.

There is a vast literature on {\em majority dynamics}, which can be
interpreted as best-response dynamics of the majority game on a graph,
which is a graphical potential game. Tamuz and
Tessler~\cite{tamuz2013majority} upper bound the number of changes in
majority dynamics on slowly growing graphs. We adapt their technique
in the proof of Theorem~\ref{thm:num-changes}.

\subsection*{Acknowledgments.}
The authors would like to thank Elchanan Mossel for some enlightening
comments.

\section{Example.}
\label{sec:example}

Before diving into the definitions and results, we provide the reader
with a simple, informal example of a graphical potential game, and of
a generalized graphical potential game. These will be games of pure
{\em network externalities}.

Let $V$ be a finite set of players located at the nodes of a graph
$G = (V,E)$. Let some of the edges in $E$ be blue and the rest be
red. Pairs of players connected by a blue edge will benefit from
choosing the same strategy, and pairs connected by a red edge will
benefit from choosing different strategies. In particular, there are
two strategies, and each player gains a unit of utility for each blue
neighbor that chooses the same strategy, and loses a unit of utility
for each red neighbor that chooses the same strategy.

This is a graphical game: the utility of each player is independent of
the choices of those who are not her neighbors.

Given a strategy profile, consider the number of blue edges along
which the two players choose the same strategy, minus the number of
blue edges along which the two players choose different strategies.
This function of strategy profiles is readily seen to be a potential
for this game, and thus this game is a graphical potential game.

We can define the same game for the case in which the set of players
is countably infinite and the graph $G$ is locally finite (i.e., each
player has a finite number of players). In this case the game clearly
still a graphical game. However, it is no longer a potential
game. Still, in some ways this game resembles a potential game, and in
particular is has a ``local potential structure''. Accordingly, it
falls into the class of generalized potential games
(Definition~\ref{def:generalized}).

An additional sense in which this game resembles a potential game is
that it has pure equilibria
(Theorem~\ref{thm:generalized-equi}). However, unlike potential games,
this game has infinite better response paths. A natural question is:
given a particular player, does there exist a better response path in
which this player makes infinitely many strategy changes? It turns out
that the answer to this question depends on the graph $G$. We show
that for some infinite graphs (for example, graphs with subexponential
growth rates) each player changes strategy only a finite number of
times, as in a finite potential game (Theorems~\ref{thm:generalized}
and~\ref{thm:generalized2}). This bound on the number of changes also
applies to large finite graphs (Theorem~\ref{thm:num-changes}), so
that, for some sequences of growing graphs, the maximum number of
strategy changes does not depend on the size of the graph.

\section{Definitions.}

\subsection{Games.}
Let $\cG$ be a game played by a finite set of players $V$, and denote
$n = |V|$. Let $\cS_i$ be the finite set of strategies available to
$i \in V$, and denote the set of strategy profiles by
$\strats = \prod_i\cS_i$. Let $u_i \colon \strats \to \R$ be player
$i$'s utility function, which maps strategy profiles to payoffs. Note
that we do not allow payoffs in $\{-\infty,\infty\}$.

\subsection{Graphical games.}
We identify $V$ with the set of nodes of a simple, undirected graph
$G=(V,E)$.  The game $\cG$ is a {\em graphical game} on $G$ if each
player's utility is a function of the strategies of her neighbors in
$G$. This can be formally expressed by the following condition on the
utility functions. Choose from each $\cS_i$ an arbitrary distinguished
strategy $o_i$. Given $a \in \strats$, denote
$a^i=(a_1,,\ldots,a_{i-1},o_i,a_{i+1},\ldots,a_n)$.

\begin{definition}
  $\cG$ is a graphical game on $G$ if
  \begin{align}
    \label{eq:graphical}
    i \neq j \mbox{ and } (i,j) \not \in E \quad \Rightarrow \quad
    \forall a \in \strats\,,\, u_i(a) = u_i(a^j).
  \end{align}
\end{definition}
It is easy to verify that this definition does not depend on the
choice of $\{o_i\}_{i \in V}$.

\subsection{Potential games.}
\begin{definition}
  $\cG$ is a {\em potential game} if there exists a (potential)
  function $\Phi \colon \strats \to \R$ such that for all $i \in V$
  and $a \in \strats$
  \begin{align}
    \label{eq:potential}
    u_i(a) - u_i(a^i) = \Phi(a) - \Phi(a^i).
  \end{align}  
\end{definition}
Note that if $\Phi$ is a potential for $\cG$ then $\Phi+C$ is also a
potential for $\cG$, for any $C \in \R$. We make a canonical choice and
assume always that
\begin{align}
  \label{eq:normalized}
  \sum_{a \in \strats}e^{\Phi(a)} = 1.
\end{align}
The rational behind this choice will become apparent below.  Here we
use the fact that utilities (and hence potentials) cannot be
infinite. It is easy to see that given a potential game, all its
potentials differ by a constant.


\subsection{Markov random fields.}
\label{sec:mrfs}
Let $G=(V,E)$ be a finite, simple, undirected graph. Associate with
each $i \in V$ a random variable $X_i$, and denote $X =
(X_1,\ldots,X_n)$. $X$ is called a {\em random field} on $G$.

Let $U$ and $W$ be subsets of $V$. A subset $A \subset V$ is a
$(U,W)$-\emph{cut} in the graph if every path from $U$ to $W$ must
pass through $A$. For any subset $Z \subset V$ define the random
variable $X_Z = \{X_i\,:\,i \in Z\}$ to be the restriction of $X$ to
$Z$.

\begin{definition}
  $X$ is a \emph{Markov random field} (MRF) if, for every $(U,W)$-cut
  $A$ it holds that conditioned on $X_A$, $X_U$ is independent of
  $X_W$.
\end{definition}

The random field $X$ is said to be {\em positive} if its distribution
is equivalent (in the sense of mutual absolute continuity) to the
product of its marginal distributions. For discrete distributions,
this is equivalent to requiring the event
$(X_1,\ldots,X_n)=(a_1,\ldots,a_n)$ to have positive probability
whenever all of the events $X_i=a_i$ have positive probability.

For positive random fields, Markov random fields can be characterized
by a weaker condition, namely that for every pair $i \neq j \in V$
such that $(i, j) \not \in E$ it holds that, conditioned on $X_{V
  \setminus \{i,j\}}$, $X_i$ is independent of $X_j$. While this seems
to be a well-known fact, we were not able to find a reference for its
proof, and therefore provide it (for finite distributions) in
Appendix~\ref{sec:pos-mrf}.

\section{Characterization of graphical potential games.}

\subsection{Decomposing graphical potential games.}

Our first result shows that the potential of a graphical potential
game can be decomposed into local potentials.

For $W \subseteq V$ and $a \in \strats$, let $a^W \in \prod_{i \in
  W}\cS_i$ be the restriction of $a$ to $W$, given by $(a^W)_i = a_i$,
for $i \in W$.

Denote by $\cC(G)$ the set of {\em maximal cliques} in $G$; these are
cliques which are not subsets of strictly larger cliques.
\begin{theorem}
  \label{claim:graphical-potential}
  Let a game $\cG$ be both a graphical game on $G$ and a potential
  game with potential $\Phi$. Then $\Phi$ can be written as
  \begin{align}
    \label{eq:potential-decomp}
    \Phi(a) = \sum_{C \in \cC(G)}\Phi^C(a^C),
  \end{align}
  for some functions $\Phi^C \colon \prod_{i \in C}\cS_i \to \R$.
\end{theorem}
The functions $\Phi^C$ are called \emph{local potentials}.

Before proving this proposition we introduce the {\em Hessian} $\Phi_{ij}$ and prove a lemma.  For $a \in \strats$ and $i,j \in V$, denote
\begin{align*}
  a^{ij} = (a_1,\ldots,a_{i-1},o_i,a_{i+1},\ldots,a_{j-1},o_j,a_{j+1},\ldots,a_n).
\end{align*}
We recall that $o_i$ is an arbitrary distinguished strategy of player $i$.
Note that $(a^i)^j = (a^j)^i = a^{ij} = a^{ji}$.

Let $\Phi_i \colon \strats \to \R$ be given by
\begin{align*}
  \Phi_i(a) = \Phi(a) - \Phi(a^i).
\end{align*}
Alternatively, by~\eqref{eq:potential}, $\Phi_i(a) = u_i(a) -
u_i(a^i)$.

The Hessian $\Phi_{ij} \colon \strats \to \R$ is given by
\begin{align*}
  \Phi_{ij}(a) = \Phi_i(a) - \Phi_i(a^j) = \Phi(a) - \Phi(a^i) -
  \Phi(a^j) + \Phi(a^{ij}).
\end{align*}

\begin{lemma}
  \label{lemma:hessian}
  Let a game $\cG$ be both a graphical game on $G$ and a potential
  game with potential $\Phi$. Then $\Phi_{ij} = 0$ for all $(i,j) \not\in E$.
\end{lemma}
\begin{proof}  
  Choose $(i,j) \not \in E$. Then by~\eqref{eq:graphical} we have that
  $u_i(a) = u_i(a^j)$ and that $u_i(a^i) = u_i(a^{ij})$. Hence
  \begin{align*}
    0 &= \left[u_i(a) - u_i(a^j)\right] - \left[u_i(a^i) - u_i(a^{ij})\right]\\
    &= \left[u_i(a)  - u_i(a^i)\right] - \left[u_i(a^j) - u_i(a^{ij})\right]\\
    &= \Phi_i(a) - \Phi_i(a^j)\\
    &= \Phi_{ij}(a).
  \end{align*}  
\end{proof}

\begin{proof}[Proof of Theorem~\ref{claim:graphical-potential}]
  Let $\mathbb{P}$ be a probability distribution over the set of
  strategy profiles $\strats$ given by
  \begin{align}
    \label{eq:p}
    \P{a} = e^{\Phi(a)}.
  \end{align}
  By~\eqref{eq:normalized} this is indeed a probability
  distribution. Let $X = (X_1,\ldots,X_n)$ be a random field on $G$
  with law $\mathbb{P}$. Then $X_i$ is a (random according to
  $\mathbb{P}$) strategy played by player $i$. Note that $X$ is a
  positive random field, although a priori it may not be Markov.

  We will prove the theorem by showing that for every $(i,j) \not \in
  E$, the random variables $X_i$ and $X_j$ are independent,
  conditioned on $\{X_k\}_{k \not \in \{i,j\}}$. This will prove that
  $X$ is a Markov random field over the graph $G$, and so the claim
  will follow immediately from the Hammersley-Clifford Theorem (see
  Hammersley and Clifford~\cite{hammersley1968markov}, and also
  Grimmett~\cite{grimmett1973theorem},
  Preston~\cite{preston1973generalized} and
  Sherman~\cite{sherman1973markov}) for positive MRFs, which states
  that $\mathbb{P}$ be can decomposed as
  \begin{align*}
    \log \P{X=a} = \sum_{C \in \cC(G)}\Phi^C(a^C),
  \end{align*}
  for some functions $\Phi^C \colon \prod_{i \in C}\cS_i \to \R$.

  Choose $(i,j) \not \in E$. Then by Lemma~\ref{lemma:hessian} we have
  that, for every $a \in \strats$,
  \begin{align*}
    \Phi(a) + \Phi(a^{ij}) = \Phi(a^i) + \Phi(a^j).
  \end{align*}
  By~\eqref{eq:p}, this can be written as
  \begin{align}
    \label{eq:independence}
    \P{X=a} \cdot \P{X=a^{ij}} = \P{X=a^i} \cdot \P{X=a^j}.
  \end{align}
  Denote by $\cX^{ij}(a) = \{X_k = a_k\}_{k \not \in \{i,j\}}$ the
  event that $X_k=a_k$ for all $k \not \in \{i,j\}$. Then we can
  write~\eqref{eq:independence} as
  \begin{align*}
    \lefteqn{\P{X_i=a_i,X_j=a_j,\cX^{ij}(a)} \cdot
      \P{X_i=o_i,X_j=o_j,\cX^{ij}(a)}}\\
    &=
    \P{X_i=o_i,X_j=a_j,\cX^{ij}(a)} \cdot
    \P{X_i=a_i,X_j=o_j,\cX^{ij}(a)}.
  \end{align*}
  Hence
  \begin{align*}
    \lefteqn{\CondP{X_i=a_i,X_j=a_j}{\cX^{ij}(a)} \cdot
      \CondP{X_i=o_i,X_j=o_j}{\cX^{ij}(a)}}\\
    &=
    \CondP{X_i=o_i,X_j=a_j}{\cX^{ij}(a)} \cdot
    \CondP{X_i=a_i,X_j=o_j}{\cX^{ij}(a)}.
  \end{align*}
  
  Summing over all possible values of $a_i$ and $a_j$ we arrive at
  \begin{align*}
    \lefteqn{\CondP{X_i=o_i,X_j=o_j}{\cX^{ij}(a)}}\\
    &=
    \CondP{X_i=o_i}{\cX^{ij}(a)} \cdot
    \CondP{X_j=o_j}{\cX^{ij}(a)}.
  \end{align*}
  Finally, we recall that the choice of $\{o_i\}_{i \in V}$ was
  arbitrary, and so $X_i$ and $X_j$ are independent, conditioned on
  $\{X_k\}_{k \not \in \{i,j\}}$.
\end{proof}

\subsection{Potentials of graphical potential games.}

Theorem~\ref{claim:graphical-potential} states that if a potential
game is graphical, then the potential has the form
\eqref{eq:potential-decomp}. We cannot hope to have the opposite
direction (i.e., if the potential function has the form
\eqref{eq:potential-decomp} then the game is graphical) because of the
following observation: If we add a constant payoff of $c$ to player
$i$ whenever player $j\neq i$ plays a certain strategy $a_j$, then the
potential function does not change; on the other hand, if in the
original game player $j$'s strategy does not influence player $i$'s
payoff (i.e., there is no edge $(i,j)$ in $G$), then in the new game
(with the additional payoff of $c$) this is no longer the case. Such a
change preserves the potential function but not the graphical
structure. Therefore, we cannot deduce a result on the graphical
structure of the game from the potential function.

However, the above mentioned change is a special type of change of the
game in the following sense: it does not cause any strategic change in
the game, because whether or not player $i$ receives an additional
payoff of $c\neq 0$ does not depend on her own behavior. This
motivates the notion of \emph{strategically equivalent games}
introduced by Monderer and Shapley~\cite{MS}. Denote by $\strats^{-i}$
the set of strategy profiles of players in $V \setminus \{i\}$.
\begin{definition}\label{def:strat-equ}
  Two games $\cG$ and $\cG'$ with payoff functions $u$ and $w$ over
  the same strategy profile set $\strats$ are called \emph{strategically
    equivalent} if there exist functions $f_1,...,f_n$ where
  $f_i:\strats^{-i}\rightarrow \mathbb{R}$ such that
  $u_i(a)=w_i(a)+f_i(a_{-i})$. We say that $\cG$ is a
  \emph{strategically equivalent transformation} of $\cG'$.
\end{definition}
Note again, that since the additional payoff of $f_i(a_{-i})$ does not
depend on player $i$'s behavior, the transformation causes no
strategic change. In particular, any (mixed) Nash equilibrium of $\cG$ is a
(mixed) Nash equilibrium of $\cG'$.

We next show that up to a strategically equivalent transformation of
the game, the decomposition~\eqref{eq:potential-decomp} is a necessary
and sufficient condition for a potential game to be graphical.

\begin{theorem}\label{pro:graphical-equivalence}
  A potential game $\cG$ with potential $\Phi$ is strategically
  equivalent to a graphical game on $G$ if and only if the potential
  $\Phi$ has the form \begin{align}
    \label{eq:potential-decomp2}
    \Phi(a) = \sum_{C \in \cC(G)}\Phi^C(a^C).
  \end{align}.
\end{theorem}

\begin{proof}
  If $\cG$ is strategically equivalent to a graphical game $\cG'$ on
  $G$, then $\cG'$ is also a potential game with the same potential
  function $\Phi$, because a strategic equivalent transformation
  preserves the potential function (see Monderer and
  Shapley~\cite{MS}). By Theorem~\ref{claim:graphical-potential}
  $\Phi$ has the form \eqref{eq:potential-decomp2}.

  For the opposite direction, let $\Phi$ be of the
  form~\eqref{eq:potential-decomp2}, and consider the game $\cG'$
  where player $i$'s payoff is defined by
  \begin{align}
    \label{eq:u-i-def}
    u_i(a) = \sum_{C \in \cC(G)\,:\,i \in C}\Phi^C(a^C).
  \end{align}
  We show that $\cG'$ is a potential game on the graph $G$ with the
  potential $\Phi$. This will complete the proof, because $\cG$ and
  $\cG'$ have the same potential, and therefore are strategically
  equivalent (see Monderer and Shapley~\cite{MS} again).

  Since $(a^i)^C = a^C$ whenever $i \not \in C$, we have that $\Phi^C(a^C) - \Phi^C((a^i)^C) = 0$ whenever $i \not \in C$. Hence
  \begin{align*}
    u_i(a) - u_i(a^i) &= \sum_{C \in \cC(G)\,:\,i \in
      C}\Phi^C(a^C)-\Phi((a^i)^C)\\
    &= \sum_{C \in \cC(G)}\Phi^C(a^C)-\Phi((a^i)^C)\\
    &= \Phi(a) - \Phi(a^i),
  \end{align*}
  where the first equality follows from~\eqref{eq:u-i-def}, the second
  equality follows from the observation above, and the
  third from the fact that $\Phi$ has the form
  \eqref{eq:potential-decomp2}.
  
  Finally, to see that $\cG$ is a graphical game over $G$, note that
  for $(i,j) \not \in E$ we have that $(a^j)^C = a^C$ for all $C \in
  \cC$ such that $i \in C$. Hence
  \begin{align*}
    u_i(a) - u_i(a^j) &= \sum_{C \in \cC(G)\,:\,i \in C}\Phi^C(a^C) -
    \Phi^C((a^j)^C)\\
    &= \sum_{C \in \cC(G)\,:\,i \in C}\Phi^C(a^C) -
    \Phi^C(a^C)\\
    &= 0.
  \end{align*}
\end{proof}

\subsection{Payoffs of graphical potential games.}

Theorem~\ref{claim:graphical-potential} characterizes the
\emph{potential function} of a graphical potential game, but it does
not characterize the \emph{payoff functions} of the players in such a
game. We next present an exact characterization of the payoff
functions in a graphical potential game. For every player $i\in V$ let
$N(i)$ be the set of neighbors of player $i$, excluding $i$. Denote by
$\strats^{N(i)}$ the set of strategy profiles of players in $N(i)$.

\begin{theorem}\label{pro:payoffs}
  Let $G$ be a graph, and let $\Phi$ be a potential of the
  form~\eqref{eq:potential-decomp}.  For every choice of functions
  $f_1,\ldots,f_n$ where $f_i:\strats^{N(i)} \rightarrow \mathbb{R}$,
  the game with the payoffs
  \begin{align}\label{eq:payoffs}
    u_i(a) = \sum_{C \in \cC(G)\,:\,i \in C}\Phi^C(a^C)+f_i(a^{N(i)})
  \end{align}
  is a graphical potential game over $G$ with the potential
  $\Phi$.

  Conversely, for every graphical potential game over $G$ with
  potential $\Phi$ there exist functions $f_1,...,f_n$ such that the
  payoffs are given by~\eqref{eq:payoffs}.
\end{theorem}

\begin{proof}
  Let $\cG$ be a game with payoffs $u_i$, and let $\cG'$ be the game
  with payoffs
  \begin{align}
    w_i(a) = \sum_{C \in \cC(G)\,:\,i \in C}\Phi^C(a^C).
  \end{align}
  $\cG'$ is a graphical potential game over $G$ with the potential
  $\Phi$; this was proved in
  Theorem~\ref{pro:graphical-equivalence}. The strategically
  equivalent transformation $u_i=w_i+f_i(a^{N(i)})$ preserves the
  potential function $\Phi$. Moreover it preserve the graphical
  structure as well, because $f_i$ depends only on the strategy of the
  neighbors of $i$. Therefore $\cG$ is a graphical potential game over
  $G$ with the potential $\Phi$.

  For the opposite direction, given a graphical potential game $\cG$
  over $G$ with the potential $\Phi$, we know that the payoffs can be
  written as:
  \begin{align}
    u_i(a) = \sum_{C \in \cC(G)\,:\,i \in C}\Phi^C(a^C)+g_i(a_{-i})
  \end{align}
  for some functions $g_i:\strats^{-i} \rightarrow \mathbb{R}$. This
  follows from the fact that the the games $\cG$ and $\cG'$ have the
  same potential function $\Phi$, and therefore are strategically
  equivalent.

  The game $\cG$ is graphical, therefore for every $j\notin N(i)\cup
  \{i\}$ it holds that $g(a_{-i})=g((a_{-i})^j)$. Otherwise, $j$ will
  have influence on $i$'s payoff. Therefore $g$ does not depend on
  $a_j$, and it can be written as $g(a_{-i})=f(a^{N(i)})$, which
  completes the proof.

\end{proof}

\subsection{Equivalence of graphical potential games and Markov random
  fields.}
In the proof of Theorem~\ref{claim:graphical-potential} we showed
how every graphical potential game on $G$ can be mapped to a Markov
random field on $G$. We next show that this map is a bijection, so
that every Markov random field on $G$ can be mapped back to the
potential of a graphical game on $G$.

Let $\mrf \subset \Delta(\strats)$ denote the set of probability
measures on $\strats$ which describe positive Markov random fields
with underlying graph $G$. Let $\pgg \subset \R^{\strats}$ be the set
of normalized (as in~\eqref{eq:normalized}) potentials of graphical
potential games over $G$ with strategy profiles in $\strats$. Define
$\psi \colon \pgg \to \mrf$ by
\begin{align}
  \label{eq:def-psi}
  [\psi(\Phi)](a) = e^{\Phi(a)}.
\end{align}
This is the same mapping that we use in the proof of
Theorem~\ref{claim:graphical-potential}.
\begin{proposition}
  The map $\psi \colon \pgg \to \mrf$ is a bijection.
\end{proposition}
\begin{proof}
  It was shown in the proof of
  Theorem~\ref{claim:graphical-potential} that the image of
  $\psi$ is indeed in $\mrf$. Since $\psi$ is clearly one-to-one, it
  remains to be shown that it is onto.

  Let $\mathbb{P} \in \mrf$ be a distribution over $\strats$ that is a
  positive MRF over $G$. Then, by the Hammersley-Clifford Theorem it
  can be written as
  \begin{align*}
    \P{a} = \prod_{C \in \cC(G)}e^{\Phi^C(a^C)},
  \end{align*}
  for some functions $\Phi^C : \prod_{i \in C}\cS_i \to \R$.

  Define a game $\cG$ on $V$ with strategies in $\strats$ by
  \begin{align*}
    u_i(a) = \sum_{C \in \cC(G)\,:\,i \in C}\Phi^C(a^C).
  \end{align*}
  In the proof of Theorem~\ref{pro:graphical-equivalence} we showed
  that $\cG$ is a graphical potential game with the potential
  \begin{align}
    \Phi(a) = \sum_{C \in \cC(G)}\Phi^C(a^C),
  \end{align}
  and therefore $\psi(\cG) = \mathbb{P}$.
  
%
%
%
\end{proof}

\section{Better-response paths in graphical potential games.}


In this section we make the simplifying assumption that utilities take
integer values, and in fact demand that each local potential is
integer\footnote{The extension to real values is straightforward if
  instead of considering better-response paths (see definition below)
  one considers $\varepsilon$-better-response paths, which are paths
  where each updating player improves her payoff by at least
  $\varepsilon$.}. Under these restrictions, we show how, in some
families of games, one can bound the total number of changes of
strategy {\em per player}, independently of the number of players in
the game.


\begin{definition}
  
A game $\cG$ with potential $\Phi$ over a graph $G$ has
\emph{$M$-bounded local potentials} if in the decomposition of the
potential \eqref{eq:potential-decomp} all the local potentials satisfy
$|\Phi^C(a^C)|\leq M$ for all $a^C$. Here we do not require $\Phi$ to
be normalized as in~\eqref{eq:normalized}. $\cG$ is called an
\emph{integral game} if all the local potentials are always integer.
\end{definition}

A \emph{better-response path} of length $L$ is defined to be a
sequence of strategy profiles $(a(t))_{t=0}^L $ such that $a(t+1)$
differs from $a(t)$ in exactly one coordinate $i=i(t)$ and $a_i(t+1)$
is a {\em strict} better-response to $a_{-i}(t)$. Player $i=i(t)$ is called
\emph{the updating player at time $t$}. Note that strict best-response
dynamics are a special case, and therefore our results will apply to
them, too.

An immediate observation regarding better-response dynamics is that in
an integral game, the length of a better-response path is bounded by
the range of the potential. This follows from the fact that in each
update the potential increases by at least one.

We analyze general better-response paths, without committing to a
particular order in which the player sequence $\{i(t)\}$ is
chosen. This analysis, in particular, applies to continuous time
better-response dynamics (or best-response dynamics) in which the path
is drawn by Poisson arrivals.

We define the {\em clique-degree} of a graph as the largest number of
maximal cliques that any vertex in the graph participates
in.\footnote{Note that in a graph of constant degree $d$, the
  clique-degree is trivially bounded by the constant $2^d$, because
  the number of cliques that a vertex $i$ participates in is bounded
  by the number of subsets of neighbors of $i$. More interestingly,
  the number of maximal cliques is sharply bounded by $4 \cdot
  3^{d/3}$, as was shown by Moon and Moser~\cite{moon1965oncliques}.}
\begin{claim}
  For every $n$-player potential game on the graph $G$ with
  $M$-bounded local potentials, where the clique-degree of $G$ is $D$,
  the potential of the game is bounded by $|\Phi(a)| < n D M $ for all
  $a \in \cS$.
\end{claim}
The claim follows from the fact that each vertex may participate in at
most $D$ cliques.

An immediate consequence is that the length of any better-response
path is at most $nDM$.  Hence, the average number of updates performed
by a player is at most $DM$, regardless of the size of the graph. The
next claim, which is the main theorem of this section, shows that for
graphs of slow enough growth (in a particular sense) we can bound by a
constant the total number of updates performed by any single player,
without having to average over all players.

Given a graph $G$, denote by $\dist(i,j)$ the graph distance between
two nodes $i$ and $j$ in $G$. This is the length of a shortest path
between $i$ and $j$. Denote by $S_r(G,i)$ the number of vertices at
exactly distance $r$ from $i$:
\begin{align*}
  S_r(G,i) = \#\{j\,:\,\dist(i,j)=r\}.
\end{align*}
\begin{theorem}
  \label{thm:num-changes}
  Let $\cG$ be an integral potential game on the graph $G$ with
  $M$-bounded local potentials, where $G$ has clique-degree
  $D$. Then the number of times that a player $k$ updates her strategy
  in any better-response path is at most
  \begin{align*}
    2DM\sum_{r=0}^\infty\left(1-\frac{1}{2DM}\right)^r S_r(G,k).
  \end{align*}
\end{theorem}

The proof appears in Section~\ref{sec:trm-proof}. Before that, we give
an example of an application, introduce a consequence of this theorem
on graphs that satisfy a ``slow enough growth'' condition, and an
example that demonstrates the tightness of this condition.

The next example is a particular case of the game introduced in
Section~\ref{sec:example}.
\begin{example}\label{ex:grid}
  Let $\mathbb{Z}^2_n$ be the $n \times n$ two dimensional grid, and
  consider the following game $\cG_n$ in which each player has two
  strategies. Label some subset of the edges blue and the rest
  red. The utility of each player is equal to the number of neighbors
  along blue edges that match her strategy, minus the number of
  neighbors along red edges that match her strategy. It is easy to see
  that $\cG$ is a graphical potential game on $\mathbb{Z}_n^2$, that
  local potentials are $1$-bounded, that the clique-degree of
  $\mathbb{Z}^2_n$ is four (since all maximal cliques are of size two
  and correspond to the edges), and that $S_r(\mathbb{Z}^2,i) \leq 4r$.

  Hence by Theorem~\ref{thm:num-changes}, in any better-response path,
  no player changes her strategy more than $1792$ times, independently
  of $n$.
\end{example}

We generalize this example to a larger class of games and graphs.

\begin{definition}\label{def:growth}
  Given a function $f:\mathbb{N}\rightarrow \mathbb{R}$ we say that
  the \emph{growth of a graph} $G$ is bounded by $f$ if for every
  vertex $i$ and every $r\in \mathbb{N}$ it holds that $S_r(G,i)\leq
  f(r)$.
\end{definition}

In order to gain some intuition about the notion of the growth of the
graph let us mention several graphs and their growths:
\begin{itemize}
\item For a line or a cycle with $n$ vertices, the growth is bounded
  by $f(r)=2$.

\item The growth of the $m$-dimensional grid is bounded by
  $f(r)=(2mr)^{m-1}$, which is \emph{polynomial in $r$} for a constant
  $m$.

\item The growth of a binary tree is at least $2^r$, and therefore is
  \emph{exponential in $r$}.
\end{itemize}

A straightforward corollary from Theorem \ref{thm:num-changes} is the following:

\begin{corollary}\label{cor:constant}
  Fix $M$, and let $G$ be a graph with click-degree $D$ and with
  growth that is bounded from above by the following exponentially
  increasing function
  \begin{align*}
    f(r)=c\left(1+\frac{1}{4DM}\right)^r
  \end{align*}
  for some constant $c>0$. Then for every potential game on the graph
  $G$ with $M$-bounded local potentials, the number of updates of
  every player along every better-response path is bounded by the
  constant $8cD^2 M^2$. In particular, this bound does not depend on the
  size of the graph.
\end{corollary}
Note that Corollary \ref{cor:constant} holds, in particular, for the
case where the graph has any subexponential growth.

\begin{proof}[Proof of Corollary \ref{cor:constant}]
  By Theorem \ref{thm:num-changes} we can bound the number of updates
  by
  \begin{align*}
    2DM\sum_{r=0}^\infty\left(1-\frac{1}{2DM}\right)^r c\left(1+\frac{1}{4DM}\right)^r \\
    \geq 2cDM\sum_{r=0}^\infty \left(1-\frac{1}{4DM}\right)^r = 8cD^2 M^2.
  \end{align*}
  
\end{proof}

The bound on the growth of the graph is crucial for the result of
Theorem \ref{thm:num-changes} (and Corollary \ref{cor:constant}), as
demonstrated by the following example of a sequence of graphs of
exponential growth, in which there is no uniform bound on the number
of strategy changes a player makes.

\begin{example}\label{ex:non-constant-changes}
  We denote by $BT_k=(V,E)$ the binary tree graph of depth $k$ and
  $n=2^{k+1}-1$ vertices. For $l=0,...,k$ we denote by
  $V_l=\{v^l_1,...,v^l_{2^l}\}$ the set of vertices at depth-level
  $l$.

  We consider the majority game $\cG$ on $BT_k$ in which each player
  has two strategies, and her utility is equal to the number of
  neighbors who played the same strategy. As in Example~\ref{ex:grid},
  it is easy to see that $\cG$ is a graphical potential game on
  $BT_k$, that local potentials are $1$-bounded, and that the
  clique-degree of $G$ is three, since again all maximal cliques
  correspond to edges. However, in this case $S_r(BT_k,i) \geq 2^r$
  for $r < k$, and so we cannot hope to use
  Theorem~\ref{thm:num-changes} to bound the number of changes without
  dependence on $n$. Indeed, we show that no such bound exists, so
  that, as $k$ grows, the number of strategy changes performed by the
  level 0 player diverges.

  Consider the following better-response path.\footnote{In fact, the
    presented path is also a best-response path.} In the initial
  configuration $a(0)$ we set the strategy of all players to be 1 at
  even depth-levels and (-1) at odd depth-levels (i.e.,
  $a(0)_v=(-1)^l$ for $v\in V_l$). The players update their strategies
  according to the following order, comprising $k$ ``waves'' of
  updates. The first wave consists of an update of the player at depth
  0 only. The second wave starts at all players at depth 1, and then
  continues with the player at depth 0. The $j$'th wave starts with
  all players at depth $j-1$, then continues with all players at depth
  $j-2$, then $j-3$ and so on, and finally ends with the player at
  depth 0. Formally, the order of updates is as follows:
  \begin{eqnarray*}
    & &\begin{cases}\text{Step 1.1: All the players in $V_0$.}
    \end{cases}\\
    & &\begin{cases} 
      \text{Step 2.1: All the players in $V_1$.}\\
      \text{Step 2.2: All the players in $V_0$.}
    \end{cases}\\
    & &\begin{cases} 
      \text{Step 3.1: All the players in $V_2$.}\\
      \text{Step 3.2: All the players in $V_1$.}\\
      \text{Step 3.3: All the players in $V_0$.}
    \end{cases}\\
    & &\vdots\\
    & &\text{Step } l.j \text{ for } 1\leq j \leq l \leq k \text{: All the players in } V_{l-j}.\\
    & &\vdots\\
    & &\begin{cases} \\
      \text{Step } k.k \text{: All the players in } V_0.
    \end{cases}
  \end{eqnarray*}
  where at each step $l.j$ the players in $V_{l-j}$ update their
  strategies in an arbitrary order.

  Note that at each step $l.j$, every player $v\in V_{l-j}$ indeed updates her strategy, because the order is designed in such a way that both $v$'s
  neighbors at depth-level $l-j+1$ play the opposite strategy, and player
  $v$ has at most three neighbors. Note also that the root player
  $v^0_1$ updates her strategy $k= \log_2(\frac{n+1}{2})$ times. Namely
  the number of her updates is \emph{not} bounded by a constant. This
  is possible on a sequence of growing binary trees because their
  growth rate is too high. As Example~\ref{ex:grid} above shows, no
  such construction is possible on a sequence of (polynomially)
  growing grids.
\end{example}

\subsection{Proof of Theorem \ref{thm:num-changes}.}\label{sec:trm-proof}
Let $G$ be a graph of clique-degree $D$ with $n$ nodes, and let $i$ be
a vertex in $G$. Let $C$ be a clique in $G$. We define $\dist(i,C)$,
the distance between $i$ and $C$, as the minimal graph distance
between $i$ and a vertex in $C$.

Given a potential $\Phi$ with $M$-bounded local potentials on $G$, let
\begin{align*}
  \lambda = 1 - \frac{1}{2DM}.
\end{align*}
Fix a vertex $k$ in $G$, and define the potential $\Theta$ by
\begin{align*}
  \Theta = \sum_{C \in \cC(G)}\lambda^{\dist(k,C)}\Phi^C.
\end{align*}

By the definition of better-response dynamics, $\Phi$ increases along
any better-response path. We claim that the same (weakly) holds for
$\Theta$. This follows from the next claim.
\begin{lemma}
  \label{clm:phi-theta}
  For every $a \in \cS$ and $b = (b_i,a_{-i}) \in \cS$ it holds
  that if $\Phi(b) > \Phi(a)$ then $\Theta(b) \geq \Theta(a)$.
\end{lemma}

In other words, the lemma claims that $\Theta$ is an \emph{ordinal
  potential} (see Monderer and Shapley~\cite{MS}) of the game $\cG$.
\begin{proof}
  Fix $a$ and $b$ such that $\Phi(b) > \Phi(a)$. Since they differ only in the strategy of $i$ then
  \begin{align}
    \label{eq:diff-phi}
    \Phi(b) - \Phi(a) = \sum_{C \in \cC(G)\,:\,i
      \in C}\big[\Phi^C(b) - \Phi^C(a)\big],
  \end{align}
  and likewise
  \begin{align}
    \label{eq:diff-theta}
    \Theta(b) - \Theta(a) = \sum_{C \in \cC(G)\,:\,i
      \in C}\lambda^{\dist(k,C)}\big[\Phi^C(b) - \Phi^C(a)\big].
  \end{align}
  Now, note that for any clique $C$ that includes $i$ it holds that either $\dist(k,C)=\dist(k,i)$ or else $\dist(k,C)=\dist(k,i)-1$. Hence, if we denote
  \begin{align*}
    \cC_1 = \{C \in \cC(G)\,:\,i \in C, \dist(k,C)=\dist(k,i)\}
  \end{align*}
  and
  \begin{align*}
    \cC_2 = \{C \in \cC(G)\,:\,i \in C, \dist(k,C)=\dist(k,i)-1\}
  \end{align*}
  then we can write~\eqref{eq:diff-theta} as
  \begin{align*}
    \frac{\Theta(b) - \Theta(a)}{\lambda^{\dist(k,C)-1}} &=
    \lambda\sum_{C\in \cC_1}\big[\Phi^C(b) - \Phi^C(a)\big] +
    \sum_{C\in \cC_2}\big[\Phi^C(b) -
    \Phi^C(a)\big]\\
    &= \sum_{C \in \cC(G)\,:\,i \in C}\big[\Phi^C(b) -
    \Phi^C(a)\big] - (1-\lambda)\sum_{C\in \cC_1}\big[\Phi^C(b) -
    \Phi^C(a)\big].
  \end{align*}
  By~\eqref{eq:diff-phi} this implies that
  \begin{align*}
    \frac{\Theta(b) - \Theta(a)}{\lambda^{\dist(k,C)-1}} \geq
    \Phi(b)-\Phi(a)-(1-\lambda)\sum_{\cC_1} \big[\Phi^C(b) -
    \Phi^C(a)\big],
  \end{align*}
  and by the definition of $\lambda$
  \begin{align*}
    \geq \Phi(b)-\Phi(a)-\frac{1}{2DM}\sum_{\cC_1} \big[\Phi^C(b) -
    \Phi^C(a)\big].
  \end{align*}
  Now, $\Phi(b) > \Phi(a)$ implies $\Phi(b) - \Phi(a) \geq 1$, since
  the potential is integer. Hence
  \begin{align*}
    \geq
    1-\frac{1}{2DM}\sum_{\cC_1}
      \big[\Phi^C(b) - \Phi^C(a)\big].
  \end{align*}
  But the summand has at most $D$ terms, each one of which is
  strictly less than $2M$ (because the local potentials are bounded by $M$).  Hence
  \begin{align*}
    \frac{\Theta(b) - \Theta(a)}{\lambda^{\dist(k,C)-1}} \geq 0,
  \end{align*}
  which completes the proof of the claim.
\end{proof}

By the definition of $\Theta$ we have that
\begin{align*}
  |\Theta(a)| \leq \sum_{C \in \cC(G)}\lambda^{\dist(k,C)}|\Phi^C(a)|.
\end{align*}
Since $|\Phi^C(a)| \leq M$ then
\begin{align*}
  |\Theta(a)| &\leq \sum_{C \in \cC(G)}\lambda^{\dist(k,C)}M\\
  &= M \cdot \sum_{r=0}^\infty\lambda^r\cdot\#\{C \in \cC(G)\,:\,\dist(k,C)=r\}.
\end{align*}
Now, the number of cliques at distance $r$ from $k$ is at most $D$
times the number of vertices at distance $r$ from $k$. Hence
\begin{align*}
  |\Theta(a)| \leq DM\sum_{r=0}^\infty\lambda^r S_r(G,k).
\end{align*}
Whenever player $k$ updates her strategy to a better one, $\Phi$
increases by at least 1, because the potential is integral. Note that
a change of strategy by player $k$ changes only the value of the local
potentials $\Phi^C$ such that $k\in C$ (i.e.,
$\dist(k,C)=0$). Therefore, an update of player $k$ causes an
increment by at least 1 (again because the potential is integral) of
the potential
\begin{align}\label{eq:dist0}
\sum_{C \in \cC(G)\,:\,\dist(k,C)=0} \Phi^C = \sum_{C \in \cC(G)\,:\,\dist(k,C)=0} \lambda^{\dist(k,C)} \Phi^C
\end{align}
and causes no change in the potential
\begin{align}\label{eq:dist1}
\sum_{C \in \cC(G)\,:\,\dist(k,C)\geq 1} \lambda^{\dist(k,C)} \Phi^C.
\end{align}
Namely, every better response of player $k$ causes an increment by at
least 1 of $\Theta$, which is the sum the potentials
in~\eqref{eq:dist0} and~\eqref{eq:dist1}.

By Lemma~\ref{clm:phi-theta} $\Theta$ is non-decreasing along a better
response path. Hence the total number of times that player $k$ updates
her strategy is at most
\begin{align*}
  2DM\sum_{r=0}^\infty\lambda^r S_r(G,k).
\end{align*}

Finally, since $\lambda = 1-1/(2DM)$, it follows that the total number
of times that player $k$ updates her strategy is at most
\begin{align*}
  2DM\sum_{r=0}^\infty\left(1-\frac{1}{2DM}\right)^r S_r(G,k).
\end{align*}
This completes the proof of Theorem~\ref{thm:num-changes}.

\section{Generalized graphical potential games.}
\label{sec:generalizations}

Let $G=(V,E)$ be a countably infinite graph, where the degree of each
node is finite.

\begin{definition}
  \label{def:generalized}
  A graphical game $\cG$ on $G$ is a {\em generalized} graphical
  potential game if it is strategically equivalent to a game whose
  utilities are given by
  \begin{align*}
    u_i(a) = \sum_{C \in \cC(G)\,:\,i \in C}\Phi^C(a^C)
  \end{align*}
  for some local potentials $\{\Phi^C\}_{C \in \cC}$.
\end{definition}
Note that $\cG$ is not necessarily a potential game, as the sum of the
local potentials may diverge. However, given the decomposition of
graphical potential games into local potentials
(Theorem~\ref{claim:graphical-potential}), we find that this is a
natural generalization. Indeed, we show that like potential games,
generalized potential games always have pure Nash equilibria. This
follows from a compactness argument.
\begin{theorem}
  \label{thm:generalized-equi}
  Every generalized potential game has a pure Nash
  equilibrium. 
\end{theorem}
\begin{proof}
  Equip $\strats = \prod_i\cS_i$ with the product topology. Fix a
  player $i$, and let $B_r = \cup_{s \leq r}S_s$ be the set of players
  at distance at most $r$ from $i$.

  For each $r > 0$, choose a strategy profile $a_r \in \strats$ such
  that every player in $B_r(i)$ is already best responding. This is
  possible, since if we arbitrarily fix the strategies of the players
  outside $B_r(i)$, the induced game on the players in $B_r(i)$ is a
  potential game, and therefore has a pure Nash equilibrium. This
  follows from the definition of generalized potential games by local
  potentials.
 
  Since $\strats$ is compact, the sequence $a_r$ will have a
  converging subsequence, with some limit $a \in \strats$. It is easy
  to see that $a$ is a Nash equilibrium. 
\end{proof}

Note that unlike finite potential games, in a generalized potential
game a better-response path does not necessarily reach an
equilibrium. However, under the appropriate conditions on the game and
growth rate of the graph, we now show that every better-response path
converges to an equilibrium.

To this end, we define the clique-degree of an infinite graph as for
finite graphs, although in this case it may diverge. Likewise, the
definition of $M$-boundedness can be used, since it only depends on
the local potentials.

The statement of the next theorem is identical to that of
Theorem~\ref{thm:num-changes}, with the exception that it refers more
widely to generalized graphical potential games. An inspection of the
proof of Theorem~\ref{thm:num-changes} will reveal that it applies
here too.
\begin{theorem}
  \label{thm:generalized}
  Let $\cG$ be an integral generalized potential game on the graph $G$
  with $M$-bounded local potentials, where $G$ has clique-degree
  $D$. Then the number of times that a player $k$ updates her strategy
  in any better-response path is at most
  \begin{align*}
    2DM\sum_{r=0}^\infty\left(1-\frac{1}{2DM}\right)^r S_r(G,k).
  \end{align*}
\end{theorem}
Note that the bound in Theorem~\ref{thm:generalized} may be infinite
(for example, in graphs of sufficiently fast exponential growth) or
finite (for example, in graph of subexponential growth), in which case
it is easy to show that it is finite for all $k \in V$. When the bound
is infinite, the statement of the theorem is vacuous. However, when
the bound is finite, it follows that in any better response path the
strategies of all players converge. We end the paper by formally
stating this observation.
\begin{theorem}
  \label{thm:generalized2}
  Let $\cG$ be an integral generalized potential game on the graph $G$
  with $M$-bounded local potentials, where $G$ has clique-degree
  $D$, and where 
  \begin{align*}
    2DM\sum_{r=0}^\infty\left(1-\frac{1}{2DM}\right)^r S_r(G,k)
  \end{align*}
  is finite for some (equivalently, all) $k \in V$. In any
  better-response path (which may be infinite), each player makes only
  a finite number of strategy changes.
\end{theorem}

\bibliography{potential}

\appendix
\section{Markov properties of positive random fields.}
\label{sec:pos-mrf}

Let $G=(V,E)$ be a finite simple graph with $|V|=n$, and let $X =
(X_1,\ldots,X_n)$ be random field on $G$. Recall that we say that $X$
is positive if $\P{X_1=a_1,\ldots,X_n=a_n} > 0$ whenever
$\P{X_i=a_i}>0$ for $i=1,\ldots,n$. We say that $X$ has the
{\em pairwise Markov property} if for all $(i,j) \not \in E$ it holds
that $X_i$ is independent of $X_j$, conditioned on $X_{V \setminus
  \{i,j\}}$. In this section we prove the following proposition.
\begin{proposition}
  \label{prop:markov-props}
  If $X$ is positive and has the pairwise Markov property then
  it is a Markov random field.
\end{proposition}

To prove this proposition we will need the following lemma.  Let
$G^{1,2}$ be the graph derived from $G$ by amalgamating the nodes $1$
and $2$ into a node $(1,2)$; the vertices of $G^{1,2}$ are $V^{1,2} =
\{(1,2), 3, \ldots, n\}$, and the edges of $G^{1,2}$ are
$$E^{1,2} = \{(i,j) \in E\,:\,i,j >2\} ~\cap~ \{((1,2),j)\,:\,j > 2,(1,j)
\in E \mbox{ or } (2,j) \in E\}. $$ Accordingly, let $X^{1,2} =
((X_1,X_2),X_3,\ldots,X_n)$ be the associated random field, where we
unite the two random variables $X_1$ and $X_2$ into the random
variable $(X_1,X_2)$, and associate it to the node $(1,2)$. Generally,
define $G^{i,j}$ and $X^{i,j}$ likewise.

\begin{lemma}
  \label{lem:amalgamate}
  When $X$ is positive and has the pairwise Markov property with
  respect to $G$, then $X^{i,j}$ is positive and has the pairwise
  Markov property, with respect to $G^{i,j}$.
\end{lemma}
\begin{proof}
  Without loss of generality, we assume that $i=1$ and $j=2$.  It is
  immediate that $X^{1,2}$ is positive.
  
  Every vertex $k$ that is connected to either $1$ or $2$ in the graph
  $G$, is connected to $(1,2)$ in $G^{1,2}$. Therefore, for proving
  the pairwise property on the graph $G^{1,2}$ we will consider a
  vertex $k$ such that neither $(1,k)$ nor $(2,k)$ is an edge in $G$,
  so that $((1,2),k)$ is not an edge in $G^{1,2}$. Without loss of
  generality, we assume that $k=3$.  We will show that, conditioned on
  $Y = X_{V\setminus\{1,2,3\}}$, $(X_1,X_2)$ is independent of
  $X_3$. This will prove the claim.

  Denote by $\cS_i$ the (finite) support of $X_i$. We would like to
  show that for all $a_1 \in \cS_1,a_2\in \cS_2$ and $a_3 \in \cS_3$ it
  holds that
  \begin{align*}
    \lefteqn{\CondP{(X_1,X_2,X_3)=(a_1,a_2,a_3)}{Y}}\\ &=
    \CondP{(X_1,X_2)=(a_1,a_2)}{Y} \cdot \CondP{X_3=a_3}{Y}.
  \end{align*}
  
  Since $X$ is positive $\CondP{X_3=a_3}{Y,(X_1,X_2)=(a_1,a_2)}$ is
  well defined, and we can apply the pairwise Markov property of $X$
  once to $X_1,X_3$ and once to $X_2,X_3$ to arrive at
  \begin{align}
    \lefteqn{\CondP{X_3=a_3}{Y,(X_1,X_2)=(a_1,a_2)}} \nonumber \\
    &= \CondP{X_3=a_3}{Y,X_1=a_1} \label{eq:using-pairwise-markov}\\
    &= \CondP{X_3=a_3}{Y,X_2=a_2}.\label{eq:using-pairwise-again}
  \end{align}
  Since $X$ is positive we can apply the same argument with any $b_2
  \in \cS_2$ substituted for $a_2$, and conclude, as
  in~\eqref{eq:using-pairwise-again}, that
  \begin{align*}
    \CondP{X_3=a_3}{Y,X_1=a_1} = \CondP{X_3=a_3}{Y,X_2=b_2}
  \end{align*}
  for all $b_2 \in \cS_2$, and that therefore
  \begin{align*}
    \CondP{X_3=a_3}{Y,X_1=a_1} = \CondP{X_3=a_3}{Y}.
  \end{align*}
  Applying~\eqref{eq:using-pairwise-markov} now yields
  \begin{align*}
    \CondP{X_3=a_3}{Y,(X_1,X_2)=(a_1,a_2)} = \CondP{X_3=a_3}{Y},
  \end{align*}
  which proves the claim.
\end{proof}

\begin{proof}[Proof of Proposition~\ref{prop:markov-props}]
  Let $A$ be a $(U,W)$-cut. We would like to show that $X_U$ is
  independent of $X_W$, conditioned on $X_A$.

  Amalgamate all the vertices in $U$ to a single vertex, and likewise
  amalgamate $A$ and $W$. The resulting graph is (a subgraph of)
  $$G'=(V',E') = (\{U,A,W\},\{(U,A),(A,W)\}),$$
  and the associated random field is precisely $(X_U, X_A, X_W)$. By
  repeated applications of Lemma~\ref{lem:amalgamate} this graph has
  the pairwise Markov property, and hence $X_U$ is independent of
  $X_W$, conditioned on $X_A$.
\end{proof}

\end{document}